\documentclass{article}

\usepackage{graphicx}
\usepackage{amsmath, amsthm}
\usepackage{amssymb,url}
\usepackage[ruled,vlined]{algorithm2e}

\newtheorem{corollary}{Corollary}
\newtheorem{openproblem}{Open Problem}
\newcommand{\OFG}{{\rm{OFG}}}

\RequirePackage[pdfusetitle]{hyperref}
\hypersetup{colorlinks,urlcolor=[rgb]{0,0,1},linkcolor=[rgb]{0,0,0},citecolor=[rgb]{0,0,0}}

\usepackage{mathtools}

\newtheorem{theorem}{\bf Theorem}[section]
\newtheorem{lemma}{\bf Lemma}[section]
\theoremstyle{definition}

\begin{document}
\title{Maximal origami flip graphs of flat-foldable vertices: properties and algorithms} 

\author{Thomas C. Hull\thanks{Western New England University, Springfield, MA, {\tt thull@wne.edu}}, Manuel Morales\thanks{Arizona State University, Tempe, AZ, {\tt mamora22@asu.edu}}, Sarah Nash\thanks{New College of Florida, Sarasota, FL, {\tt sarah.nash18@ncf.edu}}, Natalya Ter-Saakov\thanks{Rutgers University, New Brunswick, NJ, \tt{nt399@rutgers.edu}} }


\maketitle


\begin{abstract}
\textit{Flat origami} studies straight line, planar graphs $C=(V,E)$ drawn on a region $R\subset\mathbb{R}^2$ that can act as crease patterns to map, or fold, $R$ into $\mathbb{R}^2$ in a way that is continuous and a piecewise isometry exactly on the faces of $C$.  Associated with such crease pattern graphs are  \textit{valid mountain-valley (MV) assignments} $\mu:E\to\{-1,1\}$, indicating which creases can be mountains (convex) or valleys (concave) to allow $R$ to physically fold flat without self-intersecting.  In this paper, we initiate the first study of how valid MV assignments of single-vertex crease patterns are related to one another via \textit{face-flips}, a concept that emerged from applications of origami in engineering and physics, where flipping a face $F$ means switching the MV parity of all creases of $C$ that border $F$.  Specifically, we study the origami flip graph $\OFG(C)$, whose vertices are all valid MV assignments of $C$ and edges connect assignments that differ by only one face flip.  We prove that, for the single-vertex crease pattern $A_{2n}$ whose $2n$ sector angles around the vertex are all equal, $\OFG(A_{2n})$ contains as subgraphs all other origami flip graphs of degree-$2n$ flat origami vertex crease patterns.  We also prove that $\OFG(A_{2n})$ is connected and has diameter $n$ by providing two $O(n^2)$ algorithms to traverse between vertices in the graph, and we enumerate the vertices, edges, and degree sequence of $\OFG(A_{2n})$.  We conclude with open questions on the surprising complexity found in origami flip graphs of this type.
\end{abstract}


\section{Introduction}

When folding a piece of paper into a flat object, the creases that are made will be straight lines.  This describes \emph{flat origami} \cite{basics}, which we formally model with a pair $(C,P)$, called the \emph{crease pattern}, where $P$ is a closed region of the plane (our paper), and the set of creases $C=(V(C), E(C))$ is a planar graph drawn on $P$ with straight line segments for the edges.  (When the exact shape of the paper $P$ is not important, we will refer to the crease pattern merely as $C$.) If there exists a mapping $f:P\to\mathbb{R}^2$ that is continuous, non-differentiable along the edges of $C$, and an isometry on each face of $C$, then we say that $(C,P)$ is \emph{(locally) flat-foldable}. Also, folded creases come in two types when viewing a fixed side of the paper:  \emph{mountain} creases, which fold away in a convex manner, and \emph{valley} creases, which are concave.  We model this with a function $\mu:E(C)\to\{-1,1\}$, called a \emph{mountain-valley (MV) assignment} for the crease pattern $C$, where 1 (respectively $-1$) represents a mountain (respectively valley) crease.  A MV assignment $\mu$ is called \emph{valid} if $\mu$ can be used to fold $C$ into a flat object without the paper intersecting itself.  

Capturing mathematically how paper self-intersection works, and how it can be avoided, to achieve \emph{global} flat-foldability is difficult.  In fact, determining if a crease pattern $(C,P)$  is globally flat-foldable is NP-hard \cite{boxpleat,Bern}, even if a specific MV assignment is already given.  In the special case where  the crease  pattern has only one vertex in the interior of $P$, called a \emph{single-vertex crease pattern} or a \emph{flat vertex fold}, determining if a MV assignment is valid is not straight-forward \cite{basics} but can be determined in linear time \cite{GFALOP}.  Indeed, there are many open questions that remain about enumerating valid MV assignments \cite{Hull2} and understanding their structure \cite{TTT}, even for very simple crease patterns.

Flat-foldability and valid MV assignments have been of interest to scientists in the study of origami mechanics and their application in constructing metamaterials \cite{Silverberg}, even in the case of single-vertex crease patterns \cite{Kang}.  A concept that has emerged from such applications is that of a \emph{face flip}, where a valid MV assignment $\mu$ is altered by switching only the mountains and valleys that surround a chosen face $F$, denoting the new MV assignment (which may or may not be valid) by $\mu_F$.  Face flips were first introduced in the literature by VanderWerf \cite{Vander} and have been utilized in applications ranging from tuning metamaterials \cite{Silverberg} to analyzing the statistical mechanics of origami tilings \cite{Assis}.

In this paper, we explore the relationships between valid MV assignments of a given crease pattern $C$ using a tool called the \emph{origami flip graph}, denoted $\OFG(C)$, which is a graph whose vertices are all valid MV assignments of $C$ and where two vertices $\mu$ and $\nu$ are connected by an edge if and only if  there exists a face $F$ of $C$ such that flipping $F$ changes $\mu$ to $\nu$ (and vice-versa, i.e., $\nu=\mu_F$).   

Origami flip graphs were introduced in \cite{faceflips}, but only in the context of origami tessellations (crease patterns that form a tiling of the plane). In the present work, we focus on flat-foldable crease patterns that have a single vertex in the paper's interior, called \textit{flat vertex folds}, with the additional requirement that the sector angles between the creases are all equal.  We denote such a crease pattern by $A_{2n}$ where $2n$ is the degree of the vertex.  In Section~\ref{sec:background}, we provide  background on flat origami and show that $\OFG(A_{2n})$ serves as a maximal ``superset" graph for flat vertex folds--if $C$ is any other flat vertex fold of degree $2n$, then $\OFG(C)$ is a subgraph of $\OFG(A_{2n})$.
In Sections~\ref{sec:aea-connected} and \ref{sec:diameter}, we prove that $\OFG(A_{2n})$ is connected using two different algorithms for finding paths in this graph, one of which further proves that the diameter of $\OFG(A_{2n})$ is $n$. In Section~\ref{sec:edges} we describe an algorithm for computing the size of $\OFG(A_{2n})$, generating a sequence that is not in the Online Encyclopedia of Integer Sequences, and find a formula for this as well as for the degree sequence of $\OFG(A_{2n})$.     We conclude with open questions and a discussion of future work.

\section{Background and maximality of $\OFG(A_{2n})$}\label{sec:background}

Let $(A_{2n},P)$ denote the crease pattern that contains only one vertex $v$ in the interior of $P$, where $v$ has degree $2n$ and the angles between consecutive creases around $v$ are all equal (to $\pi/n$).  We normally let $P$ be a disc with $v$ at the center.  Let $e_1,\ldots, e_{2n}$ denote the creases in $A_{2n}$ and $\alpha_i$ denote the face between $e_i$ and $e_{i+1}$ (with the indices taken cyclically, so $\alpha_{2n}$ is between $e_{2n}$ and $e_1$).

A basic result from flat origami theory is \emph{Maekawa's Theorem}, which states that, if $v$ is a vertex in a flat-foldable crease pattern with valid MV assignment $\mu$, then the difference between the number of mountain and valley creases at $v$ under $\mu$ must be two, often denoted by $M-V=\pm 2$ \cite{basics}.  However, in the case of the crease pattern $A_{2n}$ Maekawa's Theorem is stronger:

\begin{theorem}[Maekawa for $A_{2n}$]\label{thm:Maekawa}
A MV assignment $\mu$ on $A_{2n}$ is valid if and only if
$$\sum_{i=1}^{2n} \mu(e_i) = \pm 2.$$
\end{theorem}

\begin{figure}
    \centering
    \includegraphics[width=\linewidth]{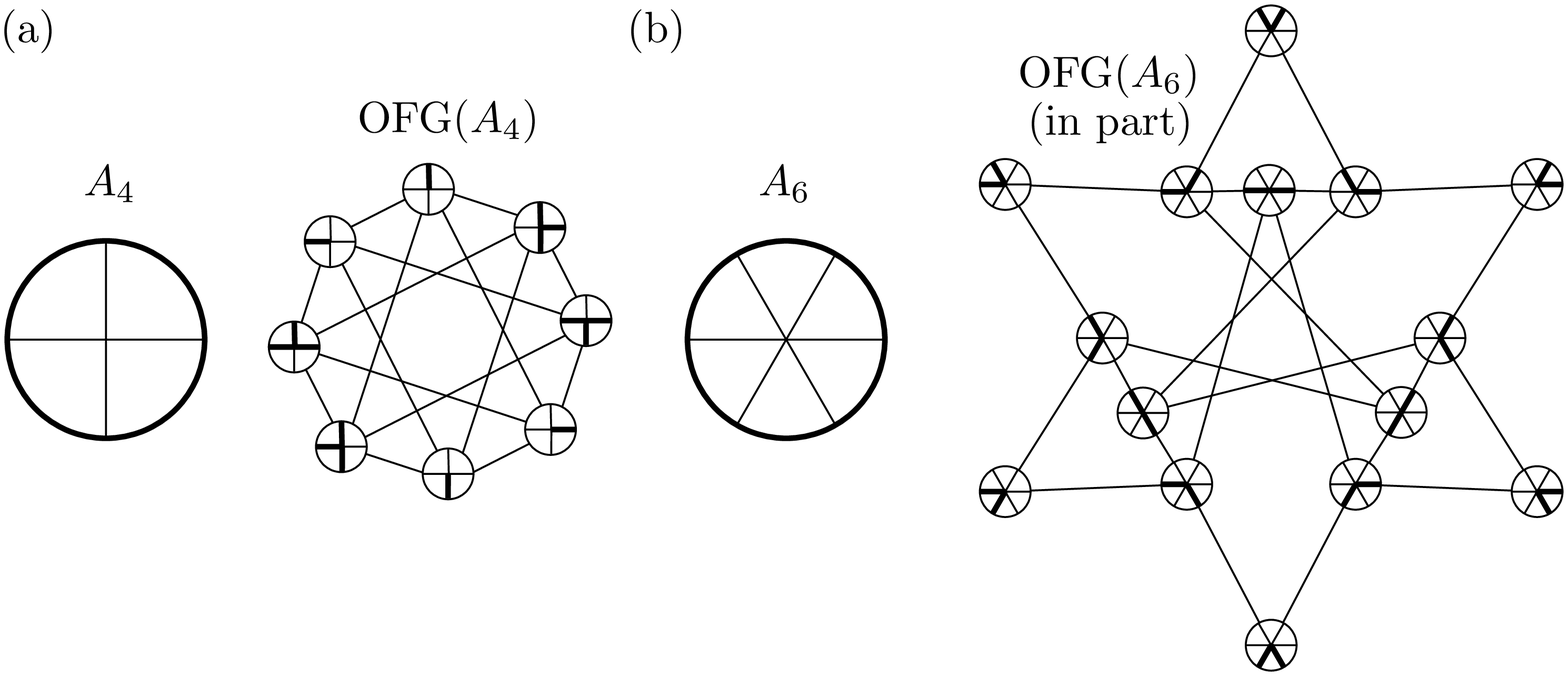}
    \caption{The crease patterns (a) $A_4$ and (b) $A_6$ along with their origami flip graphs (for $\OFG(A_6)$ only the vertices with $M-V=-2$ are shown).  Each vertex is labeled with the valid MV assignment to which it corresponds (bold/non-bold means mountain/valley, respectively).}
    \label{fig:1}
\end{figure}

Theorem~\ref{thm:Maekawa} is proved in \cite{faceflips,Origametry}, but a summary of the sufficient direction is: Find a pair of consecutive creases $e_i, e_{i+1}$ in $A_{2n}$ with $\mu(e_i)\not= \mu(e_{i+1})$ and fold them (making a ``crimp") to turn the paper into a cone, on which we now have the crease pattern $A_{2(n-1)}$ and a MV assignment that still has $\sum \mu(e) = \pm 2$.  Repeat this process until there are only two creases left, which must both be mountains or both be valleys.  This gives us a flat folding of the original vertex $A_{2n}$.

Examples of the origami flip graphs $\OFG(A_4)$ and $\OFG(A_6)$ are shown in Figure~\ref{fig:1}, although in the latter case only half of the vertices (those whose MV assignment satisfies $\sum \mu(e) = -2$) are shown.  The vertices in these graphs are labeled with their corresponding valid MV assignment, where bold creases are mountains and non-bold means valley, a convention we will use throughout this paper.  In \cite{ffconnected}, it is proved that $\OFG(C)$ is bipartite whenever $C$ is a flat-foldable, single-vertex crease pattern, although we will not be making particular use of that here.

Theorem~\ref{thm:Maekawa} tells us that any MV assignment of $A_{2n}$ that satisfies $M-V=\pm 2$ will be valid.  Therefore, there are $2\binom{2n}{n-1}$ vertices in $\OFG(A_{2n})$.

We will now show that the origami flip graph of $A_{2n}$ has maximal size over all origami flip graphs of flat vertex folds of degree $2n$, and further that such origami flip graphs are all subgraphs of $\OFG(A_{2n})$.  The idea is that when all the sector angles of a flat vertex fold are equal, the only requirement for a MV assignment to be valid is that it satisfies Maekawa's Theorem.  If, on the other hand, the sector angles are not all equal, then other restrictions will apply.  For example, if a flat-foldable, single-vertex crease pattern $C$ has consecutive sector angles $\alpha_{i-1}, \alpha_i, \alpha_{i+1}$ where $\alpha_i$ is strictly smaller than both $\alpha_{i-1}$ and $\alpha_{i+1}$, then the creases $e_i$ and $e_{i+1}$ bordering $\alpha_i$ must have different MV parity, so $\mu(e_i)\ne\mu(e_{i+1})$ must hold in any valid MV assignment $\mu$ of $C$.  (This is known as the Big-Little-Big Lemma; see \cite{basics}.)  This implies that the faces $\alpha_{i-1}$ and $\alpha_{i+1}$ can \textit{never} be individually flipped under a valid MV assignment $\mu$, since doing so would make $\mu(e_i)=\mu(e_{i+1})$.  Other restrictions on when faces in a single-vertex crease pattern can be flipped are detailed in \cite{ffconnected}, but since $A_{2n}$ does not have such restrictions, its origami flip graph will have the most edges possible.  Examples of this when $2n=4$ are shown in Figure~\ref{fig:1.5}. We formalize and slightly expand this in the following Theorem.

\begin{figure}
    \centering
    \includegraphics[width=\linewidth]{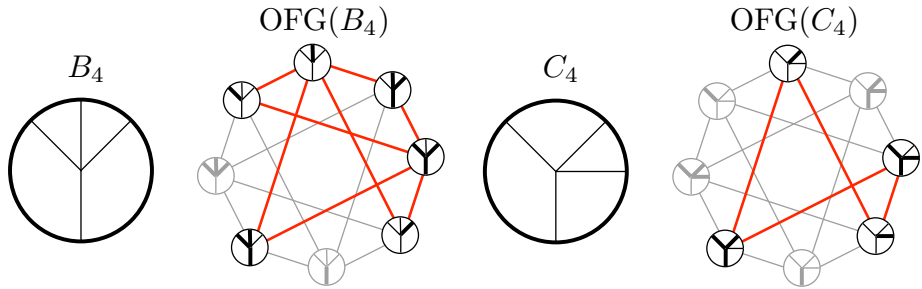}
    \caption{Other flat vertex folds $B_4$ and $C_4$ of degree 4 and their origami flip graphs, viewed as subgraphs of $\OFG(A_4)$.}
    \label{fig:1.5}
\end{figure}

\begin{theorem}\label{thm:maximal}
Let $C$ be a flat-foldable, single-vertex crease pattern of degree $2n$ that is not $A_{2n}$.  Then $\OFG(C)$ is isomorphic to at least $2n$ distinct subgraphs of $\OFG(A_{2n})$.
\end{theorem}
\begin{proof}
Suppose we have an arbitrary flat vertex fold $C$ with degree $2n$, creases $c_1, \ldots, c_{2n}$, and angles $\beta_i$. Let $\nu$ be a valid MV assignment of $C$. Then, $\nu$ also represents a valid MV assignment for $A_{2n}$.  Specifically, if $e_1,\ldots,e_{2n}$ are the creases in $A_{2n}$ and we define $\mu$ by $\mu(e_i)=\nu(c_i)$, then $\mu$ will be a valid MV assignment on $A_{2n}$ by Theorem~\ref{thm:Maekawa} (since $\nu$ must satisfy Maekawa's Theorem).  

Thus, we have a mapping $f$ between all MV assignments $\nu$ of $C$ and some MV assignments $\mu$ of $A_{2n}$ ($f(\nu) = \mu$). If $\{\nu, \nu_{\beta_i}\}$ is an edge of $\OFG(C)$ (where $\beta_i$ is flipped to make this edge), then $\{f(\nu), f(\nu_{\beta_i})\}$ is an edge of $\OFG(A_{2n})$. That is,  $\nu(c_i) = -\nu_{\beta_i}(c_i)$ and $\nu(c_{i+1}) = -\nu_{\beta_i}(c_{i+1})$ and $\nu(c) = \nu_{\beta_i}(c)$ for all other creases $c$ of $C$. The same relationship holds true between $f(\nu)$ and $f(\nu_{\beta_i})$. That is, flipping the corresponding face $\alpha_i$ (between $e_i$ and $e_{i+1}$ in $A_{2n}$) in $f(\nu)$ will result in $f(\nu_{\beta_i})$. This can be written as $f(\nu)_{\alpha_i} = f(\nu_{\beta_i})$, which implies that $\{f(\nu), f(\nu_{\beta_i})\}$ is an edge of $\OFG(A_{2n})$. Therefore, $\OFG(C)$ is isomorphic to a subgraph of $\OFG(A_{2n})$.

Furthermore, our labeling of the creases $e_i$ in $A_{2n}$ was arbitrary, and by the rotational symmetry of $A_{2n}$ we had $2n$ different ways we could have done this, resulting in at least $2n$ distinct copies of $\OFG(C)$ (since $C\ne A_{2n}$) that may be found in $\OFG(A_{2n})$.
\end{proof}

If $\mu$ is a valid MV assignment for a crease pattern $C$, then we say that a face $F$ of $C$ is \emph{flippable under $\mu$} if $\mu_F$ is also a valid MV assignment for $C$. 
In what follows, we will make extensive use of the following Lemma.

\begin{lemma}\label{lem:flip}
Let $\mu$ be a valid MV assignment of $A_{2n}$.  Then a face $\alpha_k$ is not flippable under $\mu$ if and only if $\mu(e_k)=\mu(e_{k+1})\not=$ sign$(\sum \mu(e_i))$.
\end{lemma}

\begin{proof}
By Theorem~\ref{thm:Maekawa},  $\mu_{\alpha_k}$ will be an invalid MV assignment if and only if $\sum \mu_{\alpha_k}(e_i)\not= \pm 2$.  This will only happen if $\mu(e_k)=\mu(e_{k+1})$ (i.e., the creases that border $\alpha_k$ have the same MV assignment under $\mu$) and this value, $\mu(e_k)$, is different from the majority of the creases in $\mu$.  For example, if $\sum \mu(e_i) = 2$ and $\mu(e_k)=\mu(e_{k+1})=-1$, then $\sum \mu_{\alpha_k}(e_i) = 6$, meaning that $\mu_{\alpha_k}$ violates Theorem~\ref{thm:Maekawa} and thus is invalid.  All other possibilities for $\mu(e_k)$ and $\mu(e_{k+1})$ preserve the MV summation invariant and thus allow $\alpha_k$ to be flippable under $\mu$.
\end{proof}

We will utilize the following definition in Section~\ref{sec:diameter}:  given two MV assignments $\mu$ and $\nu$ of $A_{2n}$, let $S(\mu,\nu)$ denote the set of creases $e_1,\ldots, e_{2n}$ with $\mu(e_i)\ne \nu(e_i)$.  This set is useful because it provides us with a quantity that is face-flip invariant.

\begin{lemma}\label{lem:S}
The parity of $|S(\mu,\nu)|$ (the size of $S(\mu,\nu)$) is invariant under face flips.  That is, if $\mu$ and $\nu$ are valid MV assignment of $A_{2n}$, then the $|S(\mu,\nu)|$ will have the same even/odd parity as $|S(\mu_{\alpha_i},\nu)|$ for any face $\alpha_i$ of $A_{2n}$.
\end{lemma}

\begin{proof}
Suppose we flip a face $\alpha_i$ of $A_{2n}$ under $\mu$.  Then we are changing the MV assignments of two creases.  This will change the size of $S(\mu,\nu)$ by either 0 (if exactly one of $e_i$ and $e_{i+1}$ is different between $\mu$ and $\nu$) or 2 (if $e_i$ and $e_{i+1}$ are both the same or both different between $\mu$ and $\nu$).  Therefore, the parity of $|S(\mu,\nu)|$ is invariant under face flips. 
\end{proof}

\section{Connectivity of $\OFG(A_{2n})$}\label{sec:aea-connected}


In this section, we present an algorithm for face-flipping between any two valid MV assignments $\mu$ and $\nu$ of $A_{2n}$.  This will prove that $\OFG(A_{2n})$ is connected.  In contrast, if $C$ is an arbitrary flat vertex fold, then $\OFG(C)$ is not always connected.  We invite the reader to verify that the degree-6 flat vertex fold with sector angles  $(45^\circ, 15^\circ, 60^\circ, 85^\circ, 75^\circ, 80^\circ)$ has two disconnected 4-cycles for its origami flip graph.  (Determining the connectivity of $\OFG(C)$ for general flat vertex folds $C$ is quite convoluted and beyond the scope of this paper; see \cite{ffconnected} for details.)

In the algorithm, we start with crease $e_1$.  If $\mu(e_1)=\nu(e_1)$, then we move on to crease $e_2$.  If $\mu(e_1)\not=\nu(e_1)$, then we would like to flip the face $\alpha_1$, since $\mu_{\alpha_1}(e_1) = \nu(e_1)$, and then continue the algorithm on crease $e_2$ comparing $\mu_{\alpha_1}$ with $\nu$.  

However, if $e_1$ and $e_2$ have the same MV assignment under $\mu$, then $\alpha_1$ might not be flippable under $\mu$ if it falls under Lemma~\ref{lem:flip}; such a $\mu$ and $\alpha_1$ are shown in Figure~\ref{fig:shwoop}.  Since $\alpha_1$ is not flippable, we move to $\alpha_2$ and check to see if it is flippable under $\mu$.  If so, then we flip it, and doing so will make $\alpha_1$ flippable (since it will no longer satisfy Lemma~\ref{lem:flip}).  Then we have $\mu_{\alpha_2,\alpha_1}(e_1)=\nu(e_1)$, and we may proceed with crease $e_2$ comparing $\mu_{\alpha_2,\alpha_1}$ and $\nu$.  If $\alpha_2$ is not flippable, then we try to flip the next face, $\alpha_3$.  Eventually we will find some face $\alpha_i$ that can be flipped (otherwise $\mu$ would be all mountain or all valley creases and violate Maekawa's Theorem) and then we can flip the sequence of faces $\alpha_i, \alpha_{i-1}, \alpha_{i-2}, \ldots, \alpha_1$.  We call this sequence of flipping faces in order to make $\mu_{\alpha_i,\ldots,\alpha_1}(e_1)=\nu(e_1)$ a \emph{shwoop}, and an example of such a shwoop is shown in Figure~\ref{fig:shwoop}.

\begin{figure}
\centerline{\includegraphics[width=\linewidth]{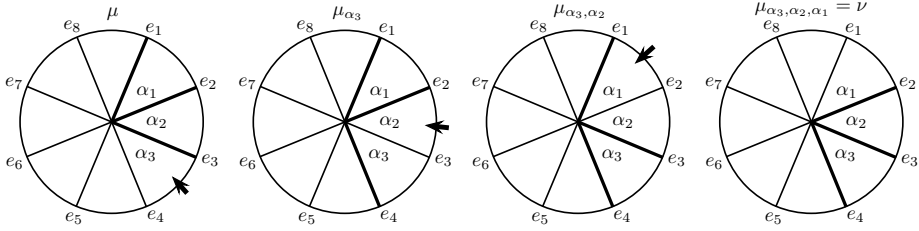}}
\caption{An example of a shwoop sequence of face flips that converts $\mu$ to $\nu$.}\label{fig:shwoop}
\end{figure}

Thus, our algorithm is to start by comparing $\mu(e_1)$ and $\nu(e_1)$, flipping $\alpha_1$ or performing a shwoop to make them agree on $e_1$ if needed, and then moving on to $e_2$, and so on.  We call this algorithm {\sc FEA-Shwoop$(A_{2n}, \mu,\nu)$}, and pseudocode for it is shown in Algorithm~\ref{alg:FEA}.  (FEA stands for Flipping Equal Angles.)

\begin{algorithm}
\DontPrintSemicolon
\centerline{{\sc FEA-Shwoop}($A_{2n}, \mu,\nu$)}
{\hrule width 0.94\linewidth}\vspace{0.8em}
Let $S=\{ \}$, $\eta=\mu$ \\
\For {$i=1$ \KwTo $2n-1$} {
   Let $m=0$\\
       \If {$\eta(e_i)\not=\nu(e_i)$} {
            \uIf {\textnormal {face $\alpha_i$ of $A_{2n}$ is flippable under $\eta$}}{
            Replace $\eta$ with $\eta_{\alpha_i}$\\
            Append $\alpha_i$ to $S$
            }
        \Else{\While{\textnormal{ face $\alpha_i$ of $A_{2n}$ is not flippable under $\eta$}}{
                 Let $m=m+1$, $i=i+1$}
             Replace $\eta$ with $\eta_{\alpha_i}$\\
            Append $\alpha_i$ to $S$\\
            \For{$j=m$ \KwTo $1$}
            {   Let $i=i-1$  \tcp{This is the shwoop.}
                 Replace $\eta$ with $\eta_{\alpha_i}$\\
            Append $\alpha_i$ to $S$
            }  
         }   }
}
Output $S$
\vspace{2mm}\noindent{\hrule width 0.94\linewidth} \vspace{1mm} 
\caption{The FEA (Flipping-Equal-Angles) Shwoop algorithm.}
\label{alg:FEA}
\end{algorithm}

\begin{theorem}\label{thm:shwoop alg}
The {\sc FEA-Shwoop$(A_{2n}, \mu,\nu)$} algorithm inputs two valid MV assignments for $A_{2n}$ and outputs a sequence of faces that, when flipped in order, will provide a sequence of valid MV assignments that start with $\mu$ and end with $\nu$.
\end{theorem}

\begin{proof}
As previously described, the algorithm uses single face-flips and shwoops to generate a sequence of valid MV assignments of $A_{2n}$ that, starting with $\mu$, make the MV parity of creases $e_1, e_2, e_3, \ldots$, in order, agree with that of $\nu$.  We need to prove that (1) finding faces to perform a shwoop is always possible and (2) that when the algorithm terminates after $i=2n-1$, the resulting MV assignment will be $\nu$. 

Suppose we are at stage $i=k$ in the algorithm where we have valid MV assignments $\mu_F$ and $\nu$ for $A_{2n}$ where $F$ is the sequence of faces we've already flipped, $\mu_F(e_i)=\nu(e_i)$ for $i=1,\ldots, k-1$, and $\mu_F(e_k)\not= \nu(e_k)$. 


Then, if $\alpha_k$ is flippable under $\mu_F$, we flip it so that $\mu_{F\cup\{\alpha_k\}}(e_k)=\nu(e_k)$ and move on to $i=k+1$.

If we cannot flip $\alpha_k$ under $\mu_F$, then that means, for example, that $\mu_F$ is majority-mountain and $e_k$ and $e_{k+1}$ are both valleys under $\mu_F$.  So we look to see if we can flip face $\alpha_{k+1}$ under $\mu_F$.  If that's not possible, then we look at face $\alpha_{k+2}$, and continue in search of a flippable face $\alpha_{k+j}$ under $\mu_F$ for some $j$. 

Suppose we get all the way to $\alpha_{2n-1}$ without finding a flippable face under $\mu_F$.  That means that $\mu_F=\nu$ on creases $e_1,\ldots, e_{k-1}$ and, assuming $\mu_F$ is majority-mountain, that $\mu_F=-1$ (valley creases) on $e_k,\ldots, e_{2n}$ (since face $\alpha_{2n-1}$ borders the creases $e_{2n-1}$ and $e_{2n}$).  Since $\mu_F$ is a valid MV assignment, this means that $\nu$ must also be all valleys on $e_k,\ldots, e_{2n}$, for if it were anything else, then $\nu$ would have fewer valley creases than $\mu_F$ and thus violate Maekawa's Theorem.  This contradicts our assumption that $\mu_F$ and $\nu$ disagreed on crease $e_k$, and so our supposition is false.  

Thus, we will find a face $\alpha_{k+j}$ that is flippable under $\mu_F$ where $k+j$ is no more than $2n-1$.  We then flip $\alpha_{k+j}$ and perform a shwoop to be able to make a new MV assignment $\mu_{F\cup\{\alpha_{k+j},\alpha_{k+j-1},\ldots, \alpha_k\}}$ that will agree with $\nu$ on crease $e_k$.  

We now examine how the algorithm terminates. The last face that could be flipped in this algorithm is $\alpha_{2n-1}$.  Let $\mu_x$ be the last MV assignment produced up to this point (so, after step $i=2n-2$ in the algorithm).  For step $i=2n-1$, suppose that $\mu_x(e_{2n-1}) = \nu(e_{2n-1})$.  This means $\mu_x$ and $\nu$ agree on all the creases $e_1, \ldots, e_{2n-1}$, which implies that they must also agree on $e_{2n}$, for otherwise one of $\mu_x$ or $\nu$ would not satisfy Maekawa's Theorem despite both being valid. Thus $\mu_x=\nu$ and the algorithm completes successfully.

Similarly, if $\mu_x(e_{2n-1})\not=\nu(e_{2n-1})$ then we must also have that $\mu_x(e_{2n})\not=\nu(e_{2n})$.  Then flipping face $\alpha_{2n-1}$ will make $\mu_{x,\alpha_{2n-1}}=\nu$, and this face-flip must be possible because $\nu$ is a valid MV assignment for $A_{2n}$.  Thus the algorithm completes successfully in this case as well.

\end{proof}

\begin{corollary}\label{cor:A2n connected}
The flip graph $\OFG(A_{2n})$ is connected.
\end{corollary}

The {\sc FEA-Shwoop$(A_{2n}, \mu,\nu)$} algorithm uses a nested loop, each of which are $O(n)$, and therefore the running time of the whole algorithm is $O(n^2)$.

\section{Diameter of $\OFG(A_{2n})$}\label{sec:diameter}

There is a different algorithm that we could use to flip between any two valid MV assignments $\mu$ and $\nu$ of $A_{2n}$, one that also proves that the diameter of $\OFG(A_{2n})$ is $n$.  We call this algorithm {\sc FEA-Halves$(A_{2n},\mu,\nu)$}.

Recall from Section~\ref{sec:background} that, if $\mu$ and $\nu$ are two valid MV assignments of $A_{2n}$, then $S(\mu,\nu)$ is the set of creases $e_i$ with $\mu(e_i)\not=\nu(e_i)$.

\begin{lemma}\label{lem:halves1}
If $\mu$ and $\nu$ are two valid MV assignments of $A_{2n}$, then $|S(\mu,\nu)|$ is even.
\end{lemma}

\begin{proof}
This can be proven using only Maekawa's Theorem by considering the sums $\sum \mu(e_i)$ and $\sum \nu(e_i)$ mod 4.  That is, these two sums are equivalent mod 4, and thus so are these sums taken only over the creases in $S(\mu,\nu)$.  But we also have $\sum_{e\in S(\mu,\nu)} \mu(e) = -\sum_{e\in S(\mu,\nu)} \nu(e)$, which implies the result.  

A more elegant proof, however, uses Lemma~\ref{lem:S} and Corollary~\ref{cor:A2n connected}.  That is, $|S(\mu,\mu)|=0$, and if we already know  that $\OFG(A_{2n})$ is connected, then since the parity of $|S(\mu,\nu)|$ is invariant under face flips, all values of $|S(\mu,\nu)|$ must be even.
\end{proof}

In lieu of Lemma~\ref{lem:halves1}, let us denote $S(\mu,\nu)=\{e_{i_1},\ldots, e_{i_{2k}}\}$, where $i_1<\cdots < i_{2k}$.  For $i<j$ let us denote $B(e_i,e_j) = \{\alpha_i,\alpha_{i+1},\ldots, \alpha_{j-1}\}$, which is the set of all faces of $A_{2n}$ between creases $e_i$ and $e_j$.  Define

$$B(\mu,\nu) = B(e_{i_1},e_{i_2})\cup B(e_{i_3},e_{i_4})\cup\cdots\cup B(e_{i_{2k-1}},e_{i_{2k}})=\bigcup_{j=1}^k B(e_{i_{2j-1}},e_{i_{2j}}).$$
That is, $B(\mu,\nu)$ is a set of faces of $A_{2n}$ between pairs of creases that have different MV parity under $\mu$ and $\nu$.  The complement set $\overline{B(\mu,\nu)}$ among the faces in $A_{2n}$ will be a similar set, and thus the sets $B(\mu,\nu)$ and $\overline{B(\mu,\nu)}$ divide the faces of $A_{2n}$ into (probably not equal-sized) ``halves."

We may now summarize the {\sc FEA-Halves} algorithm:  Find a flippable face $\alpha_{j_1}\in B(\mu,\nu)$.  We then claim that $B(\mu_{\alpha_{j_1}},\nu)$ will equal $B(\mu,\nu)\setminus\{\alpha_{j_1}\}$, and so we repeat, finding a flippable face $\alpha_{j_2}\in B(\mu_{\alpha_{j_1}},\nu)$, and so on, producing an ordering $\alpha_{j_1}, \alpha_{j_2}, \ldots$ of all the faces in $B(\mu,\nu)$ that, when flipped in order, will convert $\mu$ to $\nu$.  

\begin{lemma}\label{lem:halves2}
For valid MV assignments $\mu$ and $\nu$ of $A_{2n}$, there exists a flippable face $\alpha_j\in B(\mu,\nu)$ such that
 $B(\mu_{\alpha_{j}},\nu)=B(\mu,\nu)\setminus\{\alpha_{j}\}$.
\end{lemma}

\begin{proof}
For a set $C$ of creases, let $M(C,\mu) = $ the number of mountain creases in $C$ under a MV assignment $\mu$ and similarly define $V(C,\mu)$ for valleys.  Assume without loss of generality that $\mu$ is majority-valley on $A_{2n}$.  Then, if $\overline{S(\mu,\nu)}$ denotes the compliment of $S(\mu,\nu)$ among the creases in $A_{2n}$, we have, by Maekawa's Theorem applied to $\mu$,
\begin{equation}\label{eq:halveslem1}
M(S(\mu,\nu),\mu)+M(\overline{S(\mu,\nu)},\mu) - V(S(\mu,\nu),\mu) - V(\overline{S(\mu,\nu)},\mu)=-2
\end{equation}
Also, since $\nu$ is valid we have
\begin{equation}\label{eq:halveslem2}
M(S(\mu,\nu),\nu)+M(\overline{S(\mu,\nu)},\nu) - V(S(\mu,\nu),\nu) - V(\overline{S(\mu,\nu)},\nu)=\pm 2.
\end{equation}
However, by definition of $S(\mu,\nu)$, we know that $M(S(\mu,\nu),\mu)=V(S(\mu,\nu),\nu)$ and $V(S(\mu,\nu),\mu)=M(S(\mu,\nu),\nu)$.  Also,
$M(\overline{S(\mu,\nu)},\mu)=M(\overline{S(\mu,\nu)},\nu)$ and
$V(\overline{S(\mu,\nu)},\mu)=V(\overline{S(\mu,\nu)},\nu)$.  Thus Equation~\eqref{eq:halveslem2} becomes
\begin{equation}\label{eq:halveslem3}
V(S(\mu,\nu),\mu)+M(\overline{S(\mu,\nu)},\mu) - M(S(\mu,\nu),\mu) - V(\overline{S(\mu,\nu)},\mu)=\pm 2.
\end{equation}
\indent {\bf Case 1: $\nu$ is majority-valley.} Then Equation~\eqref{eq:halveslem3} will have $-2$ on its right-hand side, and subtracting this from Equation~\eqref{eq:halveslem1} gives
\begin{equation}\label{eq:halveslem4}
M(S(\mu,\nu),\mu) - V(S(\mu,\nu),\mu) = 0.
\end{equation}
Suppose that there is a face $\alpha_j\in B(\mu,\nu)$ whose creases $e_j$ and $e_{j+1}$ have different MV parity under $\mu$, and therefore $\alpha_j$ is a flippable face under $\mu$.  If $e_j$ or $e_{j+1}$ are in $S(\mu,\nu)$, then $S(\mu_{\alpha_j},\nu)$ will be either $S(\mu,\nu)\setminus\{e_j,e_{j+1}\}$ or $(S(\mu,\nu)\setminus\{e_j\})\cup\{e_{j+1}\}$ or $(S(\mu,\nu)\setminus\{e_{j+1}\})\cup\{e_j\}$, and so $B(\mu_{\alpha_j},\nu)$ will equal $B(\mu,\nu)$ but with the face $\alpha_j$ removed, as desired.  If neither $e_j$ nor $e_{j+1}$ are in $S(\mu,\nu)$, then they will be elements of $S(\mu_{\alpha_j},\nu)$, but, by definition of $B(\mu,\nu)$, this means that $\alpha_j$ will not be an element of $B(\mu_{\alpha_j},\nu)$, and so $B(\mu_{\alpha_j},\nu)=B(\mu,\nu)\setminus\{\alpha_j\}$.  

On the other hand, if there is no face $\alpha_j\in B(\mu,\nu)$ with $\mu(e_j)\not=\mu(e_{j+1})$, then by Equation~\eqref{eq:halveslem4} there must be a face $\alpha_j\in B(\mu,\nu)$ with $\mu(e_j)=\mu(e_{j+1})=-1$ (both valleys, since they can't all be mountains), in which case, $\alpha_j$ is flippable by Lemma~\ref{thm:Maekawa}.  Then, $\alpha_j$ must be in some component $B(e_{i_k},e_{i_{k+1}})$ in $B(\mu,\nu)$ that has only valley creases under $\mu$, whereby $B(\mu_{\alpha_k},\nu)=B(\mu,\nu)\setminus \{\alpha_j\}$.


{\bf Case 2: $\nu$ is majority-mountain.} Then, Equation~\eqref{eq:halveslem3} will have $+2$ on its right-hand-side, and subtracting from Equation~\eqref{eq:halveslem1} gives

\begin{equation*}\label{eq:halveslem5}
M(S(\mu,\nu),\mu) - V(S(\mu,\nu),\mu) = -2.
\end{equation*}
This means that we have at least two valley creases in $S(\mu,\nu)$ under $\mu$.  Let $e_{i_m}\in S(\mu,\nu)$ be a valley crease under $\mu$, and let $\alpha_j$ be the face in $B(\mu,\nu)$ that borders $e_{i_m}$.  We claim that  $\alpha_j$ is a flippable face under $\mu$:  If the other crease bordering $\alpha_j$ is also a valley under $\mu$, then since $\mu$ is majority-valley, $\mu_{\alpha_j}$ will be majority-mountain and still satisfy Maekawa's Theorem.  If the other crease bordering $\alpha_j$ is a mountain under $\mu$, then $\mu_{\alpha_j}$ is still majority-valley and satisfies Maekawa because $\mu$ did.  In both cases we have that $\mu_{\alpha_j}$ is a valid MV assignment.  Then $B(\mu_{\alpha_j},\nu)$ will have one fewer face than $B(\mu,\nu)$, the missing face being $\alpha_j$, and the Lemma is proved.

\end{proof}

\begin{algorithm}
\DontPrintSemicolon
\centerline{{\sc FEA-Halves}($A_{2n}, \mu,\nu$)}
{\hrule width 0.94\linewidth}\vspace{0.8em}
Let $L=B(\mu,\nu)$, $S=\{ \}$, $\eta=\mu$ \\
 \If {\textnormal{Length}$(L)>n$} {Let $L=$ the complement of $B(\mu,\nu)$ in $A_{2n}$}
Let $m=$\textnormal{Length}$(L)$\\
\For {$i=1$ \KwTo $m$} {
   Find $\alpha\in L$ such that $L$ is flippable under $\eta$\\
   Append $\alpha$ to $S$\\
   Replace $\eta$ with $\eta_\alpha$ and $L$ with $L\setminus\{\alpha\}$
    
}
Output $S$
\vspace{2mm}\noindent{\hrule width 0.94\linewidth} \vspace{1mm} 
\caption{The FEA (Flipping-Equal-Angles) Halves algorithm.}
\label{alg:FEA-halves}
\end{algorithm}

Therefore, the {\sc FEA-Halves} algorithm (see Algorithm~\ref{alg:FEA-halves}) will input two valid MV assignments, $\mu$ and $\nu$ for $A_{2n}$ and compute the set of faces $B(\mu,\nu) = \bigcup_{j=1}^k B(e_{i_{2j-1}},e_{i_{2j}})$ as well as the complement set of faces (in $A_{2n}$) $\overline{B(\mu,\nu)} = B(e_{i_{2k}},e_{i_1})\cup \bigcup_{j=1}^{k-1} B(e_{i_{2j}},e_{i_{2j+1}})$.  Since these form a disjoint union of all the faces in $A_{2n}$, one of $B(\mu,\nu)$ and $\overline{B(\mu,\nu)}$ will have size $\leq n$.  Pick that set, say it's $B(\mu,\nu)$, and apply Lemma~\ref{lem:halves2} repeatedly to generate a sequence of at most $n$ face flips that will transform $\mu$ into $\nu$.  This proves most of the following theorem.

\begin{theorem}\label{thm:diameter}
The flip graph $\OFG(A_{2n})$ is connected and has diameter $n$.
\end{theorem}

\begin{proof}
To see that the diameter of $\OFG(A_{2n})$ equals $n$, let $\mu$ be any valid MV assignment of $A_{2n}$ and consider the complement MV assignment $\overline{\mu}$ which is $\mu$ but with all the mountains and valleys reversed.  To transform $\mu$ to $\overline{\mu}$, every crease needs to be flipped, and (since there are $2n$ creases and each face flip switches two creases) doing this this requires at least $n$ face flips. The {\sc FEA-Halves} algorithm guarantees at most $n$ face flips to flip from $\mu$ to $\overline{\mu}$, so the diameter of $\OFG(A_{2n})$ is $n$. Examples where this can be done in $n$ face flips can be readily found (for example, let $\mu$ have $\mu(e_i)=1$ for $i=1, 3, 5, \ldots, 2n-3$ and $\mu(e_i)=-1$ for $i=2, 4, 6, \ldots, 2n$ and $i=2n-1$).
\end{proof}

Like {\sc FEA-Shwoop}, the {\sc FEA-Halves}$(A_{2n}, \mu,\nu)$ algorithm runs in $O(n^2)$ time since each pass through $B(\mu,\nu)$ to search for a flippable face takes $O(n)$ steps and Length$(B(\mu,\nu))$ is $O(n)$.

\section{Counting edges of $\OFG(A_{2n})$}\label{sec:edges}

We saw in Section~\ref{sec:background} that $\OFG(A_{2n})$ has $2\binom{2n}{n-1}$ vertices. Counting the edges in $\OFG(A_{2n})$ is not as straight-forward.  
We first perform this enumeration using the method shown in Algorithm~\ref{alg:edges1}, which we call {\sc Edge-Count}$(n)$.  This takes each valid MV assignment $\mu$ of $A_{2n}$ and uses Lemma~\ref{lem:flip} to compute the degree of $\mu$ in $\OFG(A_{2n})$: each vertex $\mu$ will have degree $2n$ unless there are non-flippable faces (bordered by ``VV" if $\mu$ is majority-mountain or by ``MM" if $\mu$ is majority-valley) which must then be subtracted from $2n$.  We then take the sum of the vertex degrees and divide by two to find the number of edges.

\begin{algorithm}
\DontPrintSemicolon
\centerline{{\sc Edge-Count}($n$)}
{\hrule width 0.94\linewidth}\vspace{0.8em}
Let $L=2\binom{2n}{n-1}$, {\sc MVAssigns} $=$ all $L$ valid MV assignments of $A_{2n}$ \\
\For{$i=1$ \KwTo $L$}
{\If{{\sc MVAssigns}$[i]$ \textnormal{is majority mountain}}{Let Deg$[i]=2n-$(number of ``VV" in {\sc MVAssigns}$[i]$)}
{\If{{\sc MVAssigns}$[i]$ \textnormal{is majority valley}}{Let Deg$[i]=2n-$(number of ``MM" in {\sc MVAssigns}$[i]$)}}}
Output $(\sum$ Deg$[i])/2$
\vspace{2mm}\noindent{\hrule width 0.94\linewidth} \vspace{1mm} 
\caption{Counting the edges in $\OFG(A_{2n})$.}
\label{alg:edges1}
\end{algorithm}

The output of {\sc Edge-Count}$(n)$ for $n=1$ to $n=13$ is
$$2, 16, 84, 400, 1820, 8064, 35112, 151008, 643500, 2722720, 11454872, 47969376, 200107544.$$
At the time of this writing, the sequence {\sc Edge-Count}$(n)$ did not appear in the Online Encyclopedia of Integer Sequences \cite{OEIS}.  

The running time of this algorithm is clearly exponential in $n$, since it checks every valid MV assignment of $A_{2n}$.  Fortunately, we can do better.

\begin{theorem}\label{thm:edgecount}
The number of edges in $\OFG(A_{2n})$ is $\frac{(n+1)(3n-2)}{2n-1}\binom{2n}{n-1}$.
\end{theorem}

Note that the formula in Theorem~\ref{thm:edgecount} matches the output of  {\sc Edge-Count}$(n)$.  We prove this formula using a probabalistic approach.

\begin{proof}
In any uniformly chosen at random MV assignment of $A_{2n}$, some faces will be flippable and some will not be flippable.  Define random variables $G =$ the number of flippable faces in a MV assignment of $A_{2n}$ (or ``good" faces) and $B=$ the number of unflippable faces (or ``bad" faces).  Also let ${\boldsymbol  1}_{\alpha_i}$ denote the indicator random variable for $\alpha_i$ being a bad face.  That is, $B={\boldsymbol  1}_{\alpha_1}+{\boldsymbol  1}_{\alpha_2}+\cdots + {\boldsymbol  1}_{\alpha_{2n}}$.
Then linearity of expectation gives us
$$\mathbb{E}[G] = \mathbb{E}[2n-B] = \mathbb{E}[2n-\Sigma{\boldsymbol 1}_{\alpha_i}] = 2n-\sum \mathbb{E}[{\boldsymbol 1}_{\alpha_i}]= 2n-2nP[\alpha_i\mbox{ is bad}].$$
Now, by Lemma~\ref{lem:flip}, $P[\alpha_i$ is bad$]= P[e_i$ and $e_{i+1}$ are minority$]=$
$$P[((e_i\mbox{ and }e_{i+1}\mbox{ are V})\mbox{ and }(\mu\mbox{ is majority M)) or }((e_i\mbox{ and }e_{i+1}\mbox{ are M})\mbox{ and }(\mu\mbox{ is majority V}))]$$ 
$$= 2P[e_i\mbox{ and }e_{i+1}\mbox{ are V and }\mu\mbox{ is majority M}] = $$
$$2P[\mu\mbox{ is majority M}] P[e_i\mbox{ and }e_{i+1}\mbox{ are V} | \mu\mbox{ is majority M}]$$
$$=2(1/2)P[e_i\mbox{ and }e_{i+1}\mbox{ are V} | \mu\mbox{ is majority M}]$$
$$= \frac{\binom{2n-2}{n-3}}{\binom{2n}{n-1}}= \frac{(n-1)(n-2)}{2n(2n-1)}.$$
Therefore $\mathbb{E}[G] = 2n(1-\frac{(n-1)(n-2)}{2n(2n-1)}.)$.  However, since MV assignments $\mu$ of $A_{2n}$ form the vertices of $\OFG(A_{2n})$, we have that $\mathbb{E}[G]=\mathbb{E}[$deg$(\mu)$ in $\OFG(A_{2n})]$, and
$$\mathbb{E}[\mbox{deg}(\mu)]=\frac{1}{|V|}\sum_{\mu\in V}\mbox{deg}(\mu)  = \frac{2|E|}{|V|}$$
where $V$ and $E$ are the vertices and edges in $\OFG(A_{2n})$, respectively.  Thus we have
$$|E|=\frac{|V|}{2} 2n\left(1-\frac{(n-1)(n-2)}{2n(2n-1)}\right) = \frac{(n+1)(3n-2)}{2n-1}\binom{2n}{n-1},$$
as desired.
 \end{proof}

The {\sc Edge-Count}$(n)$ algorithm can also be used to generate the degree sequence for $\OFG(A_{2n})$.  
Let $f_k(2n)$ denote the number of vertices of degree $k$ in $\OFG(A_{2n})$, so that {\sc Edge-Count}$(n)=(1/2)\sum_{k=n-2}^{2n} k f_k(2n)$.  The values for $f_k(2n)$ for $2\leq n \leq 6$ and the possible degrees $k$ are shown in Table~\ref{table1}, and studying these led to the following formula.

\begin{table}
\centerline{
\begin{tabular}{c|cccccccccccc}
$2n$\textbackslash$k$ & 4 & 5 & 6 & 7 & 8 & 9 & 10 & 11 & 12\\
\hline
4 & 8 \\
6 & & 12 & 18 \\
8 & & & 16 & 64 & 32\\
10 & & & & 20 & 150 & 200 & 50\\
12 & & & & & 24 & 288 & 720 & 480 & 72
\end{tabular}}
\caption{Values for $f_k(2n)$ generated by running {\sc Edge-Count}$(n)$.}\label{table1}
\end{table}

\begin{theorem}\label{thm:fk(2n)}
The number of vertices of degree $k$ in $\OFG(A_{2n})$ is
$$f_k(2n) = \frac{4n}{n+1}\binom{n+1}{k-n-1}\binom{n-2}{k-n-2},$$
for $n+2\leq k \leq 2n$.
\end{theorem}

We provide a combinatorial proof of this result developed by Jonah Ostroff.

\begin{proof}
We will enumerate the number of valid MV assignments $\mu$ of $A_{2n}$ that are majority-mountain with $b$ non-flippable faces; such a vertex in $\OFG(A_{2n})$ will have degree $k = 2n-b$, and this enumeration will equal $f_k(2n)/2$.  In this situation we will have $n+1$ mountains, $n-1$ valleys, and by Lemma~\ref{lem:flip} there should be exactly $b$ pairs of consecutive creases around $A_{2n}$ that are ``VV" under $\mu$.  That means there are exactly $n-b-1$ valley creases that are \textit{not} followed by a valley (say, going clockwise around the vertex).  Therefore we are counting the number of ways to arrange mountains and valleys so that there are exactly $n-b-1$ runs of consecutive valleys.

We can construct such MV assignments as follows:

\begin{itemize}

\item First we place the $n+1$  mountains around a circle and mark one of them as the ``start" point.

\item Then we place boxes in $n-b-1$ of the $n+1$ spaces between the mountains.  

\item Place one valley in each of the $n-b-1$ boxes.  Then place the remaining $b$ valleys in any of the $n-b-1$ boxes; by a ``stars and bars" counting argument there are $\binom{n-b-1+b-1}{b} = \binom{n-2}{b}$ ways to do this.

\end{itemize}

This gives us a MV assignment with the required conditions, but we've only counted ones that ``start" with a mountain crease.  Call the set of these MV assignments $A$.  We rotate each member of $A$ around the $A_{2n}$ crease pattern to get a bigger set of MV assignments, $B$, with $2n|A|$ elements.  We claim that each MV assignment we are looking for (valid, majority-mountain with exactly $b$ non-flippable faces) appears in $B$ exactly $(n+1)$ times.  To see this, let $\mu$ meet our required conditions and suppose $\mu$ has no rotational symmetry (meaning that each rotation of $\mu$ in $A_{2n}$ is a MV assignment distinct from $\mu$).  Then a rotated version of $\mu$ will appear in $A$ exactly $(n+1)$ times, since there are $(n+1)$ mountains in $\mu$.  These rotations of $\mu$ in $A$ will result in exactly $(n+1)$ copies of $\mu$ appearing in $B$.

On the other hand, suppose $\mu$ has rotational symmetry, say $r^j(\mu)=\mu$ for some $j$ that divides $2n$, where $r(\mu)$ is $\mu$ rotated by $\pi/n$ in $A_{2n}$.  Let $2n=qj$.  Then a rotated copy of $\mu$ will appear in $A$ exactly $(n+1)/q$ times (that is, it would be $(n+1)$ times, one for each mountain in $\mu$, but every $q$th one is a duplicate because of the rotational symmetry).  Each of these rotated copies of $\mu$ are rotated a full $2n$ times in $B$, each giving us $q$ copies of $\mu$ in $B$.  That's a total of $q(n+1)/q=(n+1)$ copies of $\mu$ in $B$.

Therefore, the number of valid MV assignments of $A_{2n}$ that are majority-mountain and have exactly $b$ non-flippable faces is

$$\frac{2n}{n+1}\binom{n+1}{n-b-1}\binom{n-2}{b}.$$
To include the majority-valley cases, we multiply by two. Substituting $b=2n-k$ and simplifying gives the desired result.

\end{proof}



\section{Conclusion}\label{sec:conclusion}

We have seen how the origami flip graph of $A_{2n}$ has the largest size among the flip graphs of flat vertex folds of degree $2n$, that it contains all such origami flip graphs as subgraphs, and that it is a connected graph with diameter $n$.  Furthermore, the algorithms used to prove these facts could be useful in further studies of origami flip graphs.  For example, the {\sc FEA-Shwoop} algorithm has the interesting property that it provides a way to flip between any two valid MV assignments of $A_{2n}$ without flipping the face $\alpha_{2n}$.  Since the labeling of the faces was arbitrary, this means that we can always avoid flipping a chosen face and still traverse the origami flip graph.  This feature is used in the forthcoming paper \cite{ffconnected} to help classify when $\OFG(C)$ will be connected for arbitrary flat vertex folds $C$.  Indeed, \cite{ffconnected} also explores when the {\sc FEA-Shwoop} algorithm can be used in other situations besides the crease pattern $A_{2n}$.

Despite $A_{2n}$ being, in a sense, the most simple case of all degree-$2n$ flat vertex folds, as it requires only Maekawa's Theorem to determine if a MV assignment will be valid, its origami flip graph nonetheless exhibits surprising complexity.  Further details on the structure of $\OFG(A_{2n})$ remains unexplored.  For instance, Theorem~\ref{thm:maximal} does not tell the whole story about the number of copies of $\OFG(C)$ that can be found in $\OFG(A_{2n})$.

\begin{openproblem}
If $C$ is a flat vertex fold of degree $2n$, how do we determine the exact number of distinct subgraphs of $\OFG(A_{2n})$ that are isomorphic to $\OFG(C)$?
\end{openproblem}

Also, we have seen that determining the degree sequence of $\OFG(A_{2n})$ involves the different ways to separate the valleys (assuming we're majority-mountain) into runs of consecutive valleys.  In other words, if we have $n-1$ valleys we are considering the integer partitions of $n-1$.  The different integer partitions affect $f_k(2n)$ for different $k$, so their influence is lost in Theorem~\ref{thm:fk(2n)}.  However, perhaps another connection is possible.

\begin{openproblem}
Can we further identify the role that integer partitions of $n-1$ play in $\OFG(A_{2n})$?
\end{openproblem}

 This is further evidence, also seen in \cite{basics}, that the single-vertex case of flat origami continues to possess more combinatorial richness than one would originally expect.


\section*{Acknowledgements}

This work was supported by NSF grants DMS-1851842 and DMS-1906202.
The authors thank Robert Dougherty-Bliss for helpful formula-conjecturing from our {\sc Edge-Count}$(n)$ data and to Jonah Ostroff for his proof of Theorem~\ref{thm:fk(2n)}.  The authors also thank Mathematical Staircase, Inc. for procuring funding for this research at the 2019 MathILy-EST REU, as well as the multitude of people who provided helpful commentary along the way. 


\bibliography{AEAff-paper.bib}
\bibliographystyle{abbrv}

\end{document}